\documentclass[12pt]{amsart}
\usepackage{amscd}
\usepackage{verbatim}
\usepackage{amssymb, amsmath, amsthm, amscd,ifthen}
\usepackage[dvips]{graphics}
\usepackage[cp866]{inputenc}
\usepackage{graphicx}
\usepackage{epsfig}
\usepackage{xcolor}

\usepackage{placeins}

\usepackage{amsmath,amssymb,amscd,}
\usepackage[all]{xy}

\theoremstyle{plain}
\newtheorem{theorem}{Theorem}
\newtheorem{lemma}{Lemma}

\newtheorem{proposition}{Proposition}

\theoremstyle{definition}
\newtheorem{definition}{Definition}

\newtheorem{remark}{Remark}

\theoremstyle{plain}

\theoremstyle{definition}
\newtoks{\thehRemark}
\newtheorem*{Remark}{\the\thehRemark}

\renewcommand{\leq}{\leqslant}
\renewcommand{\geq}{\geqslant}
\newcommand{\Rz}{\mathbb{R}}

\begin{document}

\title[Squared distance function on the configuration space of  a planar spider ]{Squared distance function on the configuration space of  a planar spider  with applications to Hooke energy and Voronoi distance}

\author{Maciej Denkowski, Gaiane Panina, Dirk Siersma}

\address{M. Denkowski: Jagiellonian University, Faculty of Mathematics and Computer Science,
Institute of Mathematics, Krak\'ow, Poland, maciej.denkowski@uj.edu.pl; G. Panina: St. Petersburg department of Steklov institute of mathematics, Russian Federation,
 gaiane-panina@rambler.ru; D. Siersma: Mathematisch Instituut, Universiteit Utrecht,  The Netherlands, d.siersma@uu.nl }



\subjclass[2020]{58K05, 70B15}

\keywords{Spider linkage, work space, critical points, Morse-Bott theory, Hooke energy, Voronoi distance}

\begin{abstract}
Spider mechanisms  are the simplest examples of arachnoid mechanisms, they are one step more complicated than polygonal linkages. Their configuration spaces have been studied intensively,
but are yet not  completely understood.
In the paper we study them using the Morse theory of the squared distance function from the "body" of the spider to some fixed point in the plane.
Generically, it is a Morse-Bott function.
We list its critical manifolds, describe them as products of polygon spaces, and derive a formula for their Morse-Bott indices.
We apply the obtained results  to Hooke energy and Voronoi distance.
\end{abstract}

\maketitle
\section{Introduction}

It has always been fruitful to study the configuration space of a  mechanical linkage using Morse or Morse-Bott theory. As an example, we would like to remind the reader that the homology groups of a polygon space in $\mathbb{R}^3$  were computed by treating the squared length of a diagonal of a flexible polygon as a Morse function in \cite{HK}. Another example:
the smoothness of polygon spaces follows from the behaviour of another squared distance function defined on the configuration space of a robot arm.

In this paper we study  a new class of mechanical linkages playing an important and increasing role in applied problems, that is,  \textit{spider mechanisms}.
At the present time the topology and control of multi-polygonal and spider mechanisms is a largely open topic. Some information about their configuration spaces can be derived from a paper of Mounoud \cite{Mou} and the results of Blanc, Shvalb, Shoram \cite{SSB,SSB2},  O'Hara \cite{Oh}, Kamiyama and Tsukada \cite{KT}.

We study planar spider mechanisms using the Morse theory of geometrically and physically motivated functions defined on the configuration space.  A planar spider consists of a  point (the \textit{body}) together with a set of \textit{legs} connected to a set of points (the \textit{feet}) which are fixed in the plane (Figure  \ref{f:spider}). Although studied intensively, their configuration spaces are not yet  completely understood.We focus on three  (mutually related) functions: the squared distance function from the body of the spider to some fixed point, the Hook energy, and the Voronoi distance function. 

\medskip

Here is the structure and the brief content of the paper.

In Section \ref{s:basic} we remind the reader basic notions related to  spider mechanisms. We define  the \textit{configuration space}, or the\textit{ spider space} of a mechanism. Its elements are all possible shapes of the given mechanism. Generically, it is a smooth manifold.
The simplest examples of spiders are \textit{robot arms}, that is, one-legged spiders, and \textit{polygonal linkages}  (two-legged spiders)  are revised.

In Section \ref{s:sqmorse} we  equip the spider space with the squared distance function from the body $X$ of the spider to some fixed point $Z$ in the plane.
Generically, it is a Morse-Bott function  on the configuration space. 
We characterize the stationary configurations, that is, we list   its critical manifolds, show, in Theorem \ref{t:critical}, that  they are equal to products of polygon spaces and derive a formula for their Morse-Bott indices   in Theorem \ref{t:indices}.  For a generic spider the stationary configurations of the squared distance  are geometrically  given by one of the following conditions: (i) the   body $X$ and the point $Z$ coincide, (ii) one leg is aligned, and the point $Z$ lies on this line, (iii) two different legs are aligned. These positions occur as critical submanifolds in the spider space (especially if the legs have more than two edges). We describe them as products of polygon spaces. 

These are  our main results; next we pass to examples and applications.

In Section \ref{s:examples} We give a number of examples of spiders.  We treat `small examples', and also give examples where the squared distance is a perfect Morse-Bott function. 

In Section \ref{s:hooke}   we  replace the (geometrically motivated) squared distance function by the (physically motivated)  \textit{Hooke energy}.
Since  the Hooke potential is up to an affine transformation equal to the squared distance of the body to the center of gravity of the feet,
we have straightforwardly  Proposition \ref{p:hooke} and
Theorem \ref{t:whooke}.

In Section \ref{s:voronoi} we treat another natural potential function, the \textit{Voronoi distance}, that is, the (squared) distance from the body of the spider to
the finite set of the feet. Description of critical points (Theorem \ref{VorCrit}) is a hybrid result which combines Voronoi and Delaunay tessellations with our results from Section \ref{s:sqmorse}. 
That is, stationary spiders of the Voronoi distance fall into two types: (i) critical spiders of the squared distance to one of the feet and (ii) configurations which project to the critical points of the Morse distance function in the plane. 

Section \ref{s:cremarks}  indicates possible generalizations.

\subsection*{Acknowledgements}

We thank Giorgi Khimshiashvili for useful discussions. This work is the result of a  joint \textit{Recherche en R\'esidence} (RER)  in CIRM (Luminy). We thank CIRM for the hospitality and excellent working atmosphere.

The research cooperation was partially funded by the program Excellence Initiative - Research at the Jagiellonian University in Krak\'ow.

D. Siersma acknowledges support from the project ``Singularities and Applications'' - CF 132/31.07.2023 funded by the European Union - NextGenerationEU - through Romania's National Recovery and Resilience Plan.

\section{Basic notions: spider mechanism, its configuration space and work space. Robot arms and polygon spaces}\label{s:basic}

\subsection{Spider mechanisms and spiders}

 We will use the convention of denoting points in ${\Rz}^2$ with capital letters, while the corresponding vectors will be indicated by lower-case letters; thus $A$ is a point, while $a$ the corresponding vector, hence the notation $||a-b||=|AB|$. We start by recalling the notion of a spider ({arachnoid mechanism}) in a $2$-dimensional space.

\begin{definition}
An $n$-legged {\it spider  mechanism} in ${\Rz}^2$  consists of \begin{enumerate}
                                                                  \item a set of $n$ fixed, pairwise distinct points (the {\it feet}) $A_1,\dots, A_n\in{\Rz}^2$, 
                                                               \item a collection of natural numbers $   p_1>1,\dots, p_n>1$, and
                                                                  \item a collection of   lengths $\ell_{i,{j+1}}>0$, $j=1,\dots, p_i$.
                                                                \end{enumerate}

An $n$-legged {\it spider} with $(p_1,\dots, p_n)$-articulated legs  in ${\Rz}^2$  is a \textit{configuration} of a given spider mechanism. It consists of 

\begin{enumerate}
  \item a centre point $X\in {\Rz}^2$ (the {\it body} of the spider that will be moving around), and
  \item the set of $n$ {\it legs} $\mathcal{L}_i$, that are each a connected chain of $p_i$ {\it edges} (rigid bars),  from $A_i$ to $X$ i.e., for each $i\in \{1,\dots, n\}$, a union of segments $[A_i^j,A_i^{j+1}]\subset{\Rz}^2$  of fixed lengths $\ell_{i,{j+1}}>0$, $j=1,\dots, p_i$ where $A_i^0=A_i$, $A_i^{p_i}=X$. The points $A_i^j$ for $j=1,\dots, p_i-1$ are called the {\it joints} of the $i$-th leg.
\end{enumerate}
One imagines a spider  mechanism with fixed feet and with the legs that are allowed to rotate freely around the joints.  We also allow intersections and self-intersections of the legs.
\end{definition}

\begin{figure}[ht]

\begin{center}
\includegraphics[width=4cm]{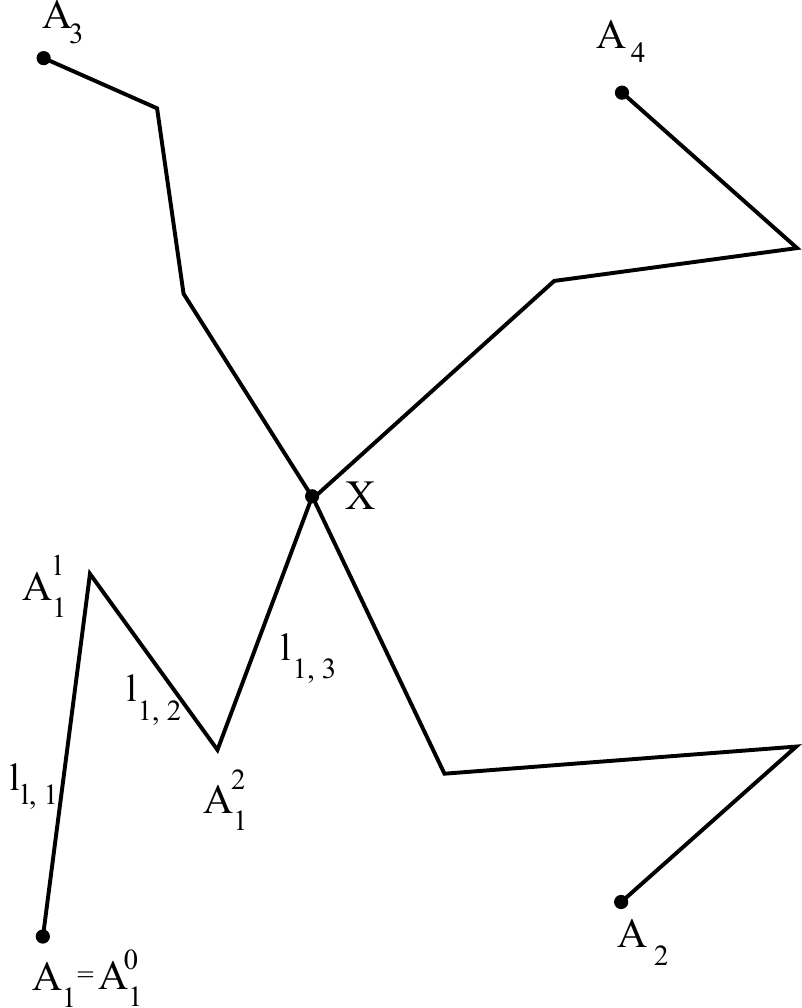} \label{f:spider}
\caption{Spider in the plane.}

\end{center}
\end{figure}
\begin{definition} The {\it configuration space} of a given spider  mechanism is the set of all the possible positions taken by its body and joints. We call it the {\it spider space} $\mathcal{S}$:
    \begin{align*}
    \mathcal{S}:=\{&(X,A_1^1,\dots,A_1^{p_1-1},A_2^1,\dots,A_n^{p_n-1})\in ({\Rz}^{2})^{1+N}\mid\\ &|A_i^jA_i^{j+1}|=\ell_{i,j+1}, i=1,\dots, n, j=0,\dots, p_i-1\},
    \end{align*}
    where $N=\sum_{i=1}^np_i-n$. It comes together with the natural projection to ${\Rz}^2$ called the {\it work map} $$\pi\colon \mathcal{S}\to \mathbb{R}^2,$$ sending a {\it spider} (i.e. a configuration) $\xi=(X,A_1,\dots, A_{N})\in \mathcal{S}$ to its body $X\in \mathcal{W}$.
The image $\mathcal{W}=\pi(\mathcal{S})$ is called the {\it work space} of the spider mechanism.

\end{definition}

The spider space and the work map have been studied in several papers  mentioned in the Introduction.

\subsection{ Robot arms: spiders with one leg}\label{ss:robotarm}
A special case is the spider with only one leg. This  mechanism is also called  a \textit{(free) robot arm}. Its underlying graph is a path graph with start point $A$ and end point $X$. It is convenient to imagine that the path graph is directed so that each edge is oriented from $A$ to $X$. The spider space of the 1-leg spider  does not depend on the edge lengths. It is  the torus $(S^1)^{p}$, where $p$ is the number of edges.  However, people usually  work with configurations of the robot arm modulo the  rotation group around  the point $A$. With this point of view the configuration space of a robot arm   is the torus $(S^1)^{p-1}$. It is convenient to imagine $(S^1)^{p-1}$ as the space of all configurations with a fixed first bar. And this is the point of view we adopt in the Proposition \ref{p:robotarm} below. 
\medskip

A configuration of a robot arm is \textit{aligned}  if it fits in a straight line.
For a generic robot arm, in an aligned configuration, each edge (geometrically) is co-directed  either with $\overrightarrow{AX}$, or with  $\overrightarrow{XA}$.
        The edges co-directed with $\overrightarrow{AX}$ are called \textit{positively directed}, the other ones  are \textit{negatively directed}.

    \begin{definition}\label{def:IndexRobArm}
       For an aligned configuration $\mathcal{L}$, denote by $Pos(\mathcal{L})$ the number of positively directed edges.
       Define also the Morse-index: $$M(\mathcal{L})=Pos(\mathcal{L})-1.$$
    \end{definition}

\begin{figure}[h]
\includegraphics[width=12cm]{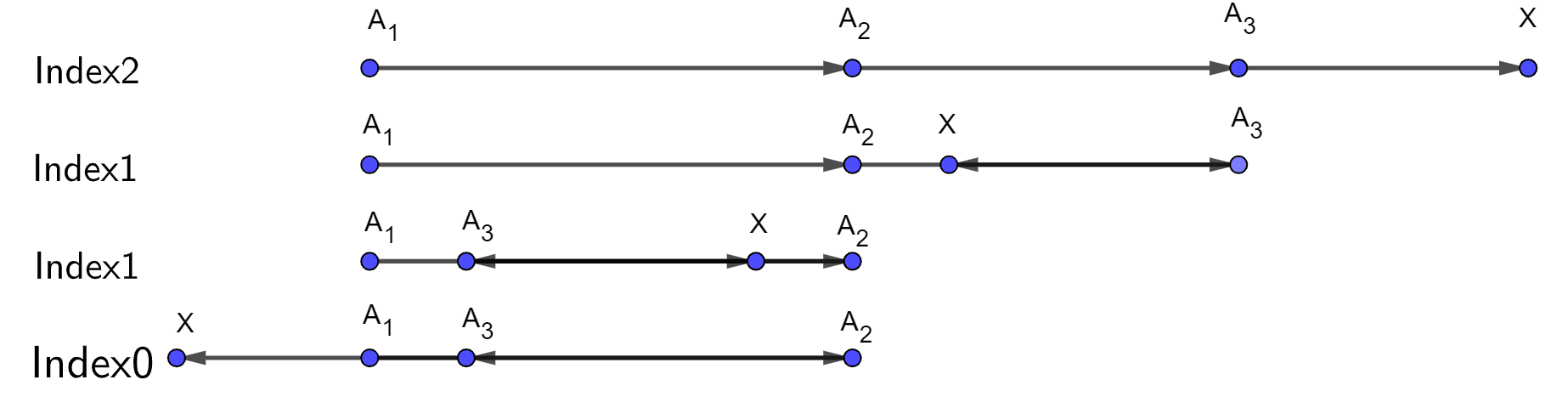}
\caption{Several values of the Morse index}
\label{fig:armindex}
\end{figure}

We are going to multiply exploit the squared distance function between the endpoints of a robot arm. In particular, we shall use the following:

\begin{proposition}\label{p:robotarm}(\cite{KM})
        The squared distance between the endpoints of a robot arm  $||a-x||^2$  (defined on the configuration space i.e. composed with the work map, cf. Section \ref{s:sqmorse}) is  a Morse-Bott function, whose critical points and critical manifolds  are:
        \begin{enumerate}
          \item isolated critical points that correspond to aligned configurations, and
          \item the set $\{\xi\in\mathcal{S}\mid ||a-\pi(\xi)||=0\}$, which exhibits the global minimum  (if it is non-empty).
        \end{enumerate}

         The Morse index of an aligned configuration  equals $M(\mathcal{L})$.
        \end{proposition}

\subsection{Polygonal linkages  and polygon spaces: spiders with two legs}
A spider with two legs reduces to a polygon space  by adding a rigid bar connecting the two feet; now  edge lengths do  matter.

A \textit{polygonal $p$-linkage} is a sequence of positive numbers $L = (l_1,...,l_p)$. We
 interpret it as a collection of $p\geq 3$ rigid bars of lengths $l_i$
joined consecutively in a closed chain by revolving joints. We always assume that the triangle
inequality holds, so the chain  can close.
A \textit{planar configuration} of $L$ is a closed broken line whose consecutive edge lengths are $l_1,...,l_p$.
As follows from the definition, a configuration may have self-intersections.

The  \textit{configuration space}, or the \textit{polygon space} for short, is the set
of all configurations of $L$ modulo orientation preserving isometries of $\mathbb{R}^2$.
It does not depend on the ordering
of $l_i$  but depends on the values of the $l_i$.

The polygon space is a closed $(p-3)$-dimensional manifold iff   there are no aligned configurations, that is, no configuration of $L$ fits in a straight line.
 There are explicit  linear conditions   in $(l_1,...,l_p)$ for the smoothness  called \textit{Grashof conditions}. For these facts and for
 more information about polygon spaces see  \cite{Farber} and  \cite{KM} .

\smallskip
For spiders with an arbitrary number of legs we have the following criterion:

\begin{proposition}[\cite{SSB}, see also \cite{H}, \cite{Mou}] \label{p:smooth1}
\begin{enumerate}
  \item  The singular points of the spider space $\mathcal{S}$ correspond to:
\begin{itemize}
    \item[-] either two aligned legs lying in one and the same straight line,
    \item[-] or three aligned legs. 
\end{itemize}
  \item Generically, the spider space  $\mathcal{S}$ is a smooth, compact, orientable, algebraic manifold  whose  dimension is $$2-2n+\sum_{i=1}^{n}p_i.$$
\end{enumerate}

\end{proposition}

\subsection{The work map}
Consider the fibres of the work map     $\pi\colon \mathcal{S}\to\mathcal{W}$. We fix the point $X$ in work space, different from the  feet points $A_i$.  We are left with the freedom of $n$ legs with fixed starting points $A_i$ and end point $X$. Each of them constitute a polygonal linkage if we add the edge $XA_i$ to $\mathcal{L}_i$ with length $l_{i,0 } = || x-a_i||$: {\em the closure of the leg}. Each of these linkages defines a polygon space. 
If $X=A_i$, the $i$-th leg itself is closed and yields a polygonal linkage. 
 It follows:

 \begin{proposition}\label{p:product}\begin{enumerate}
                                       \item The fibre of the work map over $X \ne A_i$ is the  product of the polygon spaces of the linkages, which are the closures of the legs with endpoint $X$.
                                       \item If $X=A_i,$  the fibre of the work map over $A_i$ equals the product of the polygon space of the corresponding linkage with a circle $S^1$ and the polygon spaces of the other leg closures.
                                     \end{enumerate}

\end{proposition}

\subsection{Discs-Annuli structure on the work space} \label{s:annuli}

Fix a foot $A_i$ and assume for the moment that  the leg attached to it is moving freely in ${\Rz}^2$. Consider for the  leg $\mathcal{L}_i$ the squared distance  between the foot and the body  $$|| a_i - x ||^2.$$  According to Proposition \ref{p:robotarm} its critical points
correspond to { aligned legs}, (i.e. all the points $A_i^j$ lie on the line ${\Rz}(x-a_i)$).
The critical points project via $\pi$ to circles with centre $A_i$ and {\it critical} radii $r_i^{\epsilon_i}$:
$$C(A_i,r_i^{\epsilon_i}) := \{z\in {\Rz}^2 \mid ||x-a_i|| = r_i^{\epsilon_i} \}.$$
with
$ r_i^{\epsilon_i} = | \sum_{j=1}^{p_i} \epsilon_{i,j} l_{i,j}| $ and $\epsilon_i = (\epsilon_{i,j})_{j=1}^{p_i}$; $\epsilon_{i,j} \in \{-1,+1\}$.  
{This construction includes the possibility that  (for a non-generic case) a critical radius is zero and the circle degenerates to a point.}

\begin{figure}

\begin{center}
 \includegraphics[width=6cm]{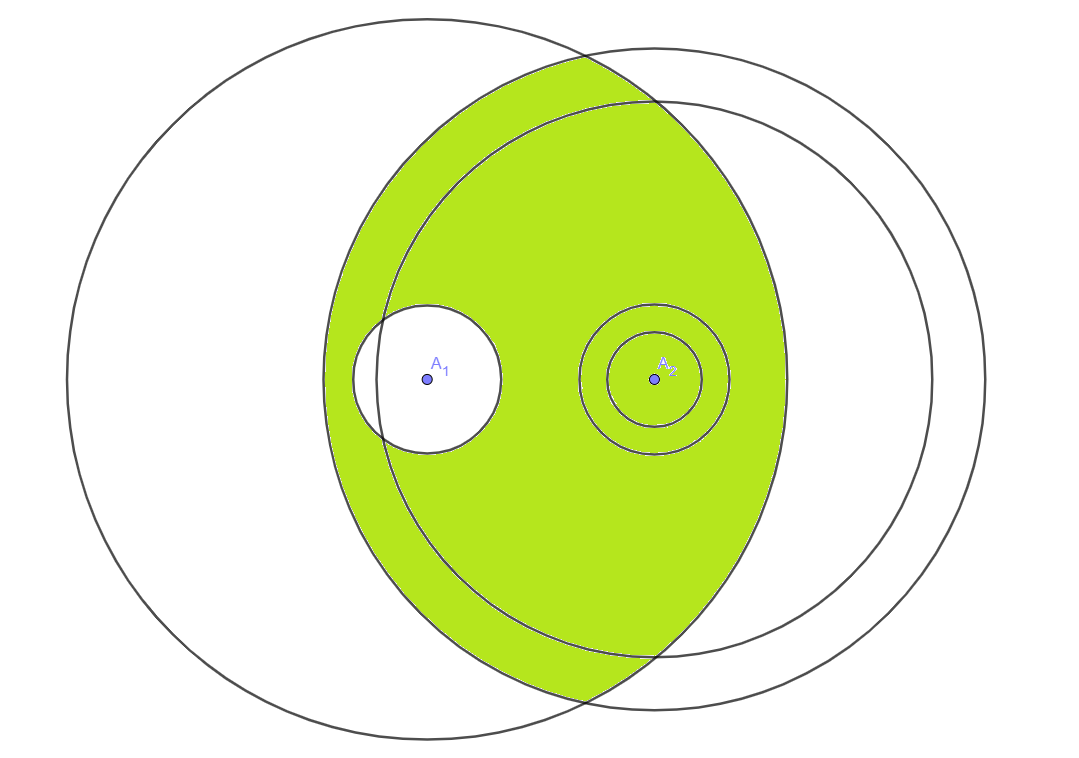}
   \caption{ Critical circles and work space.}
\end{center}
\end{figure}
We denote by $r_i$, respectively $R_i$, the minimal and the maximal values of $||x-a_i||$. They define the zone:
$$\mathcal{Z}(A_i):=\{x\in {\Rz}^d\mid \rho_i\leq ||x-a_i||\leq R_i\},$$
which is either a disc or an annulus with center $A_i$.

Here  $$\rho_i=\left\{
                                                         \begin{array}{ll}
                                                           r_i, & \hbox{if the leg cannot close;} \\
                                                           0, & \hbox{otherwise.}
                                                         \end{array}
                                                       \right.
$$ 
The zone  $\mathcal{Z}(A_i)$ is the work space of  the free leg.

\medskip

The following three facts have obvious proofs.
\begin{proposition}The work space of a spider linkage with $n$ legs is the intersection of all the zones:
    $\mathcal{W}=\bigcap_{i=1}^n \mathcal{Z}(A_i)$.
\end{proposition}

\begin{proposition}
Assume $\mathcal{S} $ is smooth. The critical locus
    of the projection $\pi\colon\mathcal{S}\to\mathcal{W}$ consists of the fibres with at least 1 aligned leg. The discriminant set coincides with those points of the work space $\mathcal{W}$, that lie on one of the circles $C(A_i,r_i^{\epsilon_i})$.
\end{proposition}

\subsection{The Euler characteristic of the spider space}\label{p:eulerchar}

The discriminant divides the work space in strata $\sigma$ of dimension 0,1 and 2 in such a way that above each stratum the fibre of $\pi$ has constant topological type.

\begin{proposition} In the situation considered, the Euler characteristic of the spider is
$$\chi (\mathcal{S} )= \sum_{\sigma}  (-1)^{dim( \sigma)} \chi (\sigma) \chi(\pi^{-1} (x_{\sigma}) )  $$
where the sum ranges over all the strata, and $x_{\sigma}$ is any point of the stratum $\sigma$.
\end{proposition} 

The proof follows from the properties of the Euler characteristic.
This formula extends the formula in \cite{Mou} for legs with only two edges to the general case. The fibres are polygon spaces. In the smooth case their Betti numbers and Euler characteristic were computed in \cite{FS}. See the tripod example in Section \ref{ss:tripod}.

\section{The squared distance function  as a Morse function} \label{s:sqmorse}
\subsection{Squared distance function defined on the spider space.}

 For a  given spider mechanism, let $Z\in {\Rz}^2$ be  a fixed point. The squared distance  from the body $X$ to $Z$ is a function defined  on the spider space. More precisely, 
set $D(x)=||x-z||^2$, for $x\in {\Rz}^2$ and $\mathbb{D}(\xi)=D(\pi(\xi))$, for $\xi \in\mathcal{S}$.
In this section we intend to study the critical point theory of $\mathbb{D}$. We will assume that the spider space $\mathcal{S}$  is smooth, i.e. satisfies the conditions mentioned in Proposition \ref{p:smooth1}, although the non-smooth case is also of interest.

\begin{definition}
A  smooth function $ f : M \to  \mathbb{R}$ on a smooth manifold $M$ is called a \textit{Morse-Bott function}  if and only if  its critical set $\Sigma (f)$ is a disjoint union of connected submanifolds, and for each connected component $C \subset \Sigma(f)$  the Hessian of the restriction of $f$ to a normal slice is non-degenerate for all $ x \in C$. The index of this Hessian is called the Morse-Bott index or transversal Morse index of the connected component $C$.
\end{definition}
For properties of Morse-Bott functions we refer to Bott \cite{Bo}, the textbook   \cite{Ni} and to Section \ref{ss:MBpol}
on the Morse-Bott  polynomial.

\smallskip

We shall show that generically, $\mathbb{D}$ is a Morse-Bott function.  The \textit{genericity}  notion now becomes more restrictive. It includes the
conditions of smoothness of the spider space, but also some more conditions on the relative position of $Z$, that will be specified later.  Anyhow, the genericity  condition defines  an open
everywhere dense subset of the parameter space of spider mechanism.

\subsection{Critical points}

Assume we have a spider mechanism with a smooth spider space $\mathcal{S}$.
  For  its leg $\mathcal{L}_i$, define its extension $z\mathcal{L}_i$ as the robot arm  obtained by adding the edge $ZA_i$ before the leg
  $\mathcal{L}_i$. See Figure \ref{fig:constructions}. This  robot arm starts with $Z$, ends   at $X$, and has $p_i + 1 $ edges.
  
  \begin{figure}[h]
  \includegraphics[width=14cm]{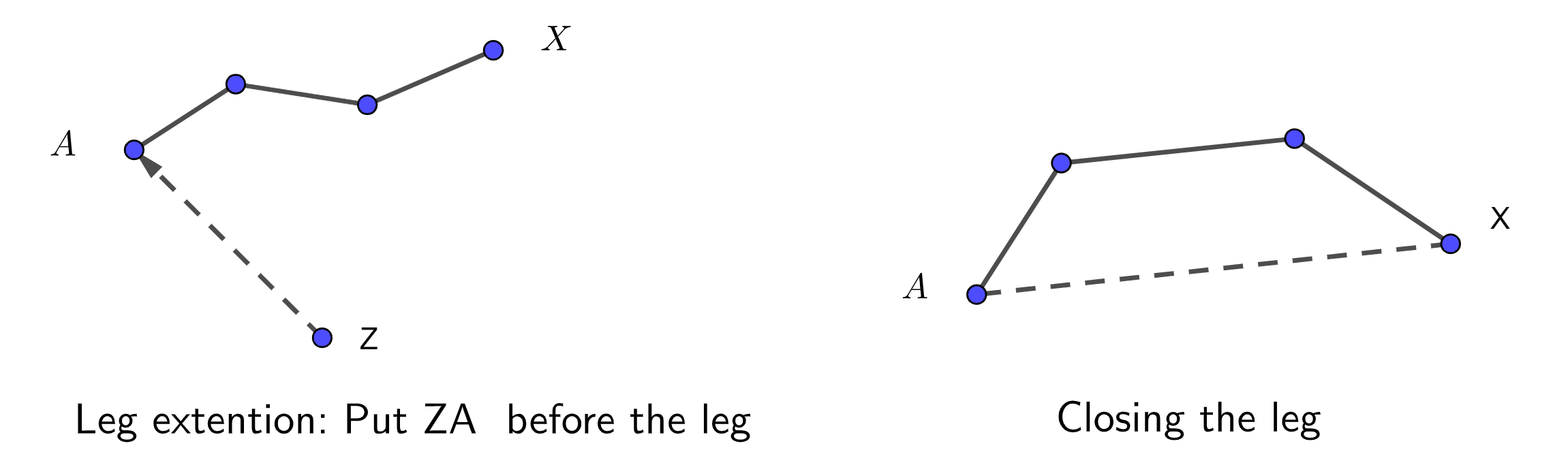}
  \caption{Constructions with legs}
  \label{fig:constructions}
  \end{figure}

Let us first consider  mechanisms with only one leg.
We assume that the following \textit{genericity condition}  holds:
\begin{itemize}
\item[$(\star)$]
 no two points $A$, $X$, and $Z$ coincide in an aligned configuration of the leg. 
\end{itemize}

 Proposition \ref{p:robotarm} implies:
\begin{proposition}\label{CorOneLeg}
 For  a  generic one-legged spider, the critical points of  $\mathbb{D}=||x-z||^2$ are:
         \begin{enumerate}

           \item
            A spider $\xi $ with $X\neq Z$  iff it is aligned, and the point $Z$ lies on the straight line  containing the spider.
           In this case the critical point is Morse and the Morse index is
           $$ M(\xi, Z) = Pos ( z\mathcal{L}) - 1$$
           In terms of the leg $\mathcal{L}$ this can be written as
           $$M(\xi, Z)=\left\{
                      \begin{array}{ll}
                       p- Pos(\mathcal{L}), & \hbox{if $X$ lies between $A$ and $Z$;} \\
                        Pos(\mathcal{L})-1, & \hbox{if $Z$ lies between $A$ and $X$;} \\
                          Pos(\mathcal{L}), & \hbox{if $A$ lies between $Z$ and $X$.}
                      \end{array}
                    \right.
           $$

            \item The (possibly empty) submanifold of $\mathcal{S}$  defined by $ X = Z$ is the global minimum manifold of $\mathbb{D}$. It is a Morse-Bott  critical submanifold with   Morse-Bott index  $0$.

         The  manifold is diffeomorphic to the configuration space of the polygonal linkage obtained by adding the bar $ZX$ to the robot arm.
         \end{enumerate}
\end{proposition}

     \begin{proof}
     (1)  The statement follows straightforwardly. One should only take into account that
     in the first case (when  $X$ lies between $A$ and $Z$) the positively directed edges  turn to negatively directed ones.
     (2) The function $\mathbb{D}$ is the composition of a submersion and the Morse function, which attains its minimum  at $X=Z$. 
     \end{proof}

Now we pass to a spider with  $n$ legs.

\begin{theorem}
\label{t:critical}
For a smooth spider space, a configuration $\xi \in \mathcal{S}$ (satisfying $(\star)$) is critical  for $\mathbb{D}$ iff  it satisfies at least one of the following conditions:
  \begin{enumerate}
        \item  $X=Z$;
        \item  One leg is aligned, and the point $Z$ lies on this line;
        \item  Two different legs are aligned.
    \end{enumerate}

\end{theorem}

\begin{proof}  

 For  $X=Z$ we clearly have the global minimum,
 so  below we assume $X \ne Z$.
 
 The derivative $ d_\xi\mathbb{D}$ of $\mathbb{D}$ is a composition of
    $$  d_\xi\pi: T_{\xi}\mathcal{S} \to T_{\pi(\xi)} \mathbb{R}^2  \; \;  \mbox{\rm and} \; \; d_xD: T_{\pi(\xi)} \mathbb{R}^2 \to \mathbb{R}.$$

   Here $d_xD$ is a submersion. If $\xi$ does not belong to the critical locus of $\pi$ then we get the composition of two submersions, so $\mathbb{D}$ is regular at $\xi$.

   As soon as $\xi $ has exactly one aligned leg, the image of $d_\xi\pi$
   is 1-dimensional (tangent line to a critical circle). The composition has maximal rank if $ZX $ is orthogonal to this image (or: the level curves of $||x-z||^2$ are tangent to the critical circle).
   As soon as $\xi $ has two non-parallel aligned legs, the image of $d_\xi\pi$ is 0-dimensional and therefore the composition is not regular.
   The conditions for smoothness from Proposition \ref{p:smooth1} allow us to restrict to  the above cases.
\end{proof}

\begin{definition}
  A pair: (a spider mechanism, a point $Z$) is called \textit{strongly  generic} if:

\begin{enumerate}
\item Its configuration space is smooth:
\begin{enumerate}
\item No configuration has more than two aligned legs.
  \item Two aligned legs are never parallel.
  \end{enumerate}
  
  \item In case of  two aligned legs, the point $Z$  never lies on any of these lines.
  \item The point $Z$  coincides with none of the feet $A_i$.
  \item For one aligned leg, $Z$ never coincides with $X$.
\end{enumerate}

\end{definition}

 \subsection{Morse indices}


Assuming strong genericity, we start with some preparations.
Let us have a deeper look at case 3  from Theorem \ref{t:critical}.  Assume that there are two aligned legs  (say, with feet at $A_1$ and $A_2$). The three points $A_1,A_2,X$ define a partition of the plane (cf. Figure \ref{fig:quadrants}) into four quadrants separated by the lines $\Rz [A_1X]$ and $\Rz [A_2X]$. Let $Z$ be in one of the quadrants and define the following signs, depending on $Z$ for the 2 aligned legs at $X$:
$$ \sigma (Z) = (\sigma_1(Z), \sigma_2(Z)) = (sign( A_1X \cdot ZX) ,sign (A_2X \cdot ZX)) $$

The next Theorem gives our final result about the Morse indices of spiders. We refer already to section \ref{s:examples} for figures, which could be helpful when following the proof.

\begin{figure}
\includegraphics[width=6cm]{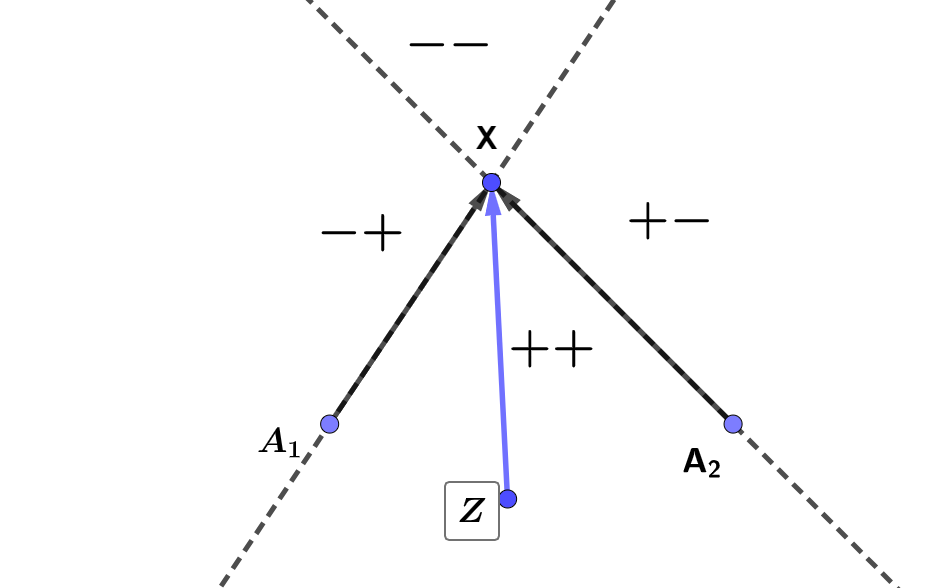}
\caption{Partition of the plane.}
\label{fig:quadrants}
\end{figure}

  \begin{theorem}
  \label{t:indices}
  For a strongly generic spider and a point $Z$, the function  $\mathbb{D}=||x-z||^2$ is a
 Morse-Bott function. The Morse-Bott indices  $\mu$ are as follows:
\begin{enumerate}
        \item The condition $X=Z$  defines a Morse-Bott critical manifold with the index $$\mu=0.$$ 
        \item One leg  $\mathcal{L}$ is aligned and $Z$ lies on this line. The Morse-Bott index of the critical manifold is $$\mu= M(\mathcal{L}, Z),$$ 
        see Proposition \ref{CorOneLeg}, where $\mathcal{L}$ is the one-legged spider obtained by eliminating all the legs except for the aligned one.

        \item  Two  legs are aligned.
       Then the Morse-Bott index  is
        $\mu=M_1+M_2,$
        where $M_i$  is defined by

          $$M_i =\left\{
                      \begin{array}{ll}
                       M(\mathcal{L}_i), & \hbox{if $\sigma_i(Z) > 0$} \\
                        p_i - M(\mathcal{L}_i), & \hbox{if $\sigma_i(Z) < 0$.}
                      \end{array}
                    \right.
           $$
           
           Here $M(\mathcal{L})$ is defined in Definition \ref{def:IndexRobArm}.
    \end{enumerate}

The critical manifolds in the three cases are  diffeomorphic to:
\begin{enumerate}
        \item  The Cartesian product of $n$ polygon spaces (obtained by adding the bar $A_iX$ to the leg $\mathcal{L}_i$),
        \item The Cartesian product of $n-1$ polygon spaces (obtained by closing non-aligned legs in a similar way),
        \item 
        The  Cartesian product of $n-2$ polygon spaces, also arising from closing of  non-aligned legs.
    \end{enumerate}

   \end{theorem}


\begin{proof} (1) is clear  since $X=Z$  defines the global minimum manifold.

(2)  By the strong genericity assumption, only  one leg of the spider is aligned. Non-aligned legs contribute to the Morse-Bott critical manifold only, which is (via closing the legs, cf. Figure \ref{fig:constructions}) the product of polygonal spaces, so we eliminate them and leave only the aligned leg.  The statement follows from Proposition \ref{p:robotarm}.

(3)  By the strong genericity assumption, only two legs of the spider are aligned, and the lines that contain the legs are different.  As in (2), we can eliminate the non-aligned legs and leave only the two aligned legs. Denote the configuration space of the two-legged spider by $\mathcal{S}'$. Its dimension is $p_1+p_2-2$.

Let us compute the $2$-jet  of $\mathbb{D}$ at the critical point. The first order terms vanish since we have a critical spider.
Set $r_i :=|A_iX| =  |A_iX_0|+ \sum_j \pm \xi_{i,j}^2$ near the aligned legs, where $X_0$ is the  body of the critical spider, that is,  intersection point of the two critical circles. Near $X_0$  we  use $(r_1,r_2)$ as coordinates  for $X$. Define the unit vectors $e_1,e_2$ at the point $X_0$, having the same directions as $A_1X_0$  and $A_2X_0$ respectively.

Denote by  $\psi$  the map from local coordinates $\xi$ to the spider space to $X= (r_1,r_2)$, and
consider the second order term in the Taylor series expansion of ${D}\circ \psi$.  Only the 2-jet of $\psi$ and the 1-jet of $D=||x-z||^2 $ at $x_0$ are important. The latter can be written as $||z-x_0|| + (z-x_0) \cdot x $.\\
The composition of the two maps gives now
$$c + (z-x_0)\cdot (x_0 +  \sum_j \pm \xi_{1,j}^2 e_1  +  \sum_j \pm \xi_{2,j}^2  e_2)=
||x_0-z|| +  \alpha_1  \sum_j  \pm \xi_{1,j}^2+ \alpha_2  \sum_j \pm \xi_{2,j}^2  $$

with $\alpha_1 = (z-x_o) \cdot e_1$ and $\alpha_2 = (z-x_o) \cdot e_1$, which yields the result required.
\end{proof}


\subsection{Morse-Bott Polynomials}\label{ss:MBpol}
In the case of isolated singularities the counting of critical points via a generating polynomial, \textit{the Morse polynomial,} relates the critical points of the squared distance $\mathbb{D}$ with the topology of $\mathcal{S}$. 

\smallskip
The Morse-Bott polynomial related to the function $\mathbb{D}$ is defined
by
$$P_{\mathbb{D}} (t)= \sum_{\Sigma} P_{\Sigma}(t) \cdot t^{i(\Sigma)} :=\sum \mu_k (\mathbb{D}) t^k   $$
where  $P_{\Sigma}(t) = \sum b_k (\Sigma) t^k$ is the Poincar\'e polynomial, $b_k(\Sigma)$ the $k^{th}$  Betti number  and
$i(\Sigma)$  the transversal Morse index of the Morse-Bott singular component $\Sigma$.
One computes the coefficients  $\mu_k$ by performing the multiplication and summation.
In the case of isolated singularities $\mu_k$ coincides with the number of Morse points.

From \cite{Bo,Ni} we recall that
$$ P_{\mathbb{D}}(t) = P _{\mathcal{S}}(t) + (1+t) Q(t)$$
where $ P _{\mathcal{S}}(t)$ is the Poincar\'e polynomial of $\mathcal{S}$ and      $Q(t)$ is a polynomial with non-negative integer coefficients.
From this follow the  Morse inequalities: $\mu_k \ge b_k$.
Note that:
\begin{itemize}
\item $\mu_k$  depends on the position of $Z$, but is  locally constant on the complement of a bifurcation set.
\item For $t=-1$ we get the Euler-Poincar\'e characteristic $\chi(\mathcal{S}) = \sum (-1)^k \mu_k$: independent of $Z$.
\end{itemize}
In the next section we will use the Morse-Bott polynomial to get information about the topology of $\mathcal{S}$, i.e. for computing  $\chi(\mathcal{S})$ and upper bounds for the Betti numbers of $\mathcal{S}$.
For this  we need to know  $P_{\Sigma}$ and also the Betti-numbers of spider spaces.  We refer to the paper of Farber and Sch\"utz \cite{FS}, where there is a complete result in terms of short, medium and long legs.
 Examples can be found in Section \ref{s:examples}.


\section{Examples}\label{s:examples}
In this section we compute critical points of the squared distance  function and their indices and compare them with topological invariants of $\mathcal{S}$.
For efficiency of the exposition we restrict to examples with a few legs only.

\subsection{A tripod}\label{ss:tripod}
Spider  mechanisms  with only two edges for each leg are easy to analyse. In strongly generic cases  $\mathbb{D}$  is a Morse function, and there are isolated critical points only.  

As an example, consider a spider whose feet are positioned at the vertices of an equilateral triangle, and the edge lengths are as is depicted in Figure   \ref{fig:tripodfig}.

\begin{figure}[ht]
\begin{minipage}[t]{0.45\linewidth}
\centering
\vspace{-6.5cm}
\includegraphics[
                 width=0.95\linewidth]{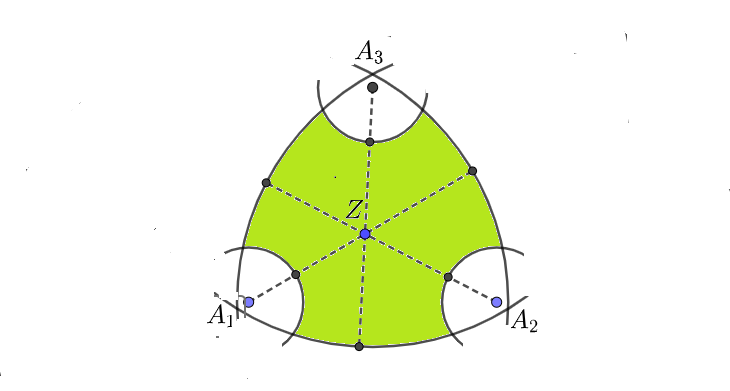}
                 (Above) The work space of the tripod.\\
                 (Right) The critical configurations of the tripod:
                 (a) minima in $Z$,\\ (b) saddles with one aligned leg, \\ (c) maxima with one aligned leg,\\ (d) saddles with two aligned legs.
\end{minipage}%
\begin{minipage}[t]{0.55\linewidth}
\begin{center}
 \includegraphics[width=6cm]{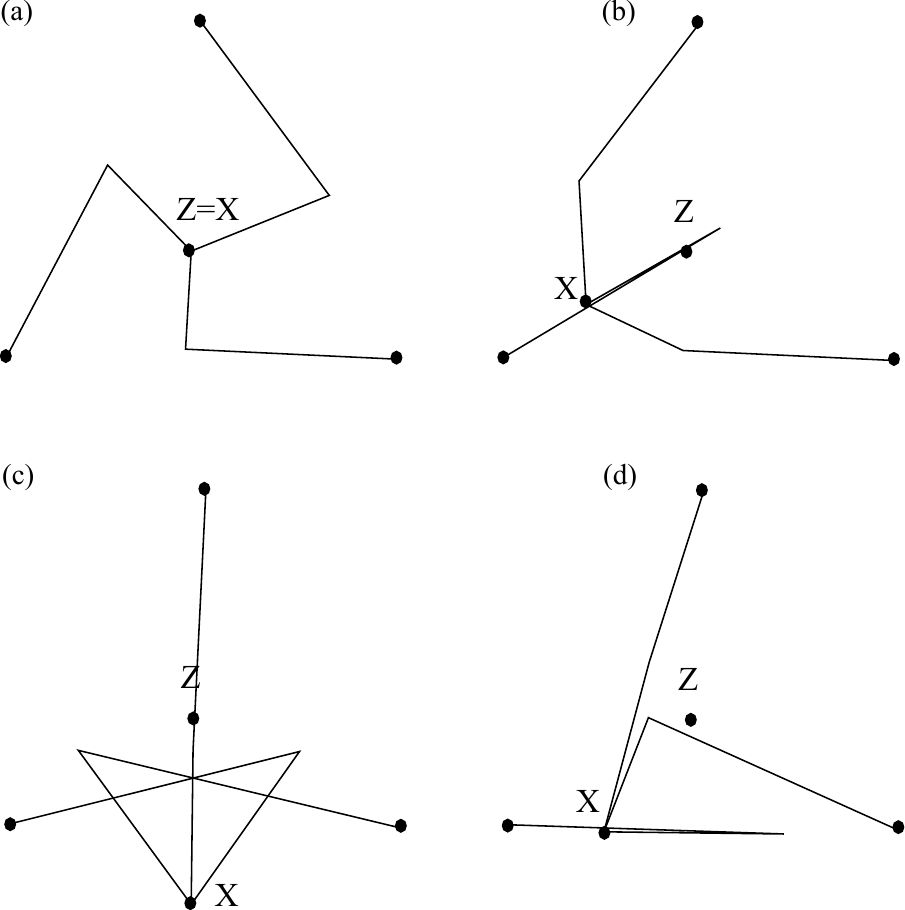}

\end{center}

\end{minipage}
\vspace{-0.5cm}
  \caption{ Work space and critical configurations.}
   \label{fig:tripodfig}
\end{figure}


The spider space $\mathcal{S}$  is a two-dimensional manifold. It is a patch of eight copies of the hexagonal work space, so its genus is $3$. See \cite{Mou} and Proposition \ref{p:eulerchar}.

Assume that $Z$ is positioned at the center of the triangle.The global minimum of $\mathbb{D}$ is attained when $X=Z$, which corresponds to eight different configurations.

The saddle points  are of two types: with one aligned leg ($12$ such points), and with two aligned legs (also $12$ ones).  

There are $12$ points where the maximum is attained: one leg is aligned.
We get $44$ isolated critical points and Euler characteristic $-4$. Other choices of $Z$ (e.g. outside the work space) will reduce the number of critical points. We will exploit this idea in the next example.

\subsection{Second Example}

We next compute Morse-Bott polynomials in the case of arbitrary number of edges for each leg. In this example we assume only two legs, with $p_1$, resp $p_2$ edges (more or less of equal length). The dimension of spider space is $p_1+p_2 -2$. We assume moreover that the work space is the region bounded by the two critical circles of maximum radius and does not contain arcs of other critical circles. We will consider two different choices of the point $Z$ (see Figures \ref{fig:second} and \ref{fig:exterior}).

The first choice is that the point $Z$ is in the interior of the work space. According to Theorem 2 there are 5 critical components,  related to the point $Z$ and the points $X_1,X_2,X_3,X_4$  in the figure. The singular components are as follows:
\begin{itemize}
\item Type  (2). For the points $X_3$ and $X_4$: each is just one point $\bullet$ i.e. an isolated singularity, clearly a maximum; 
\item Type(1). For the points $X_1$ and $X_2$: the polygon space of the two polygons, containing $X_1$, respectively $X_2$ and the points $A_2$, respectively $A_1$. Their Poincar\'e  polynomials can be computed by the formula given in \cite{FS}. But one can also see directly that they are homeomorphic to spheres $S^{p_2-2}$, respectively $S^{p_1-2}$ since the linkages  contain one very long  leg in comparison with the other legs (we omit the details). 

\item Type (0). For the point $Z$: the product of the two polygon spaces  $S^{p_2-2} \times S^{p_1-2}$. 
\end{itemize}
The indices and the contributions to the Morse-Bott  polynomial are contained in the table in Figure \ref{fig:second}.
\vspace{-0.3cm}

\begin{figure}[ht]
\begin{minipage}[t]{0.3\linewidth}
\centering
\includegraphics[
                 width=0.95\linewidth]{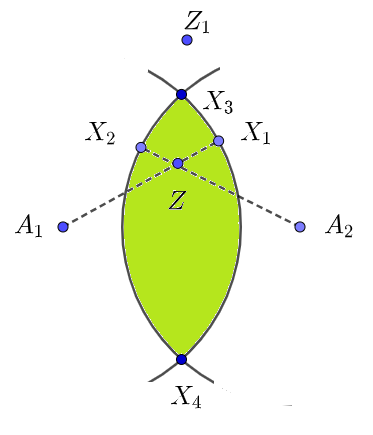}
\end{minipage}%
\begin{minipage}[t]{0.65\linewidth}
\vspace{-3.5cm}
\begin{tabular}{|c|ccc|}
\hline
            & $index$ & $\Sigma$ & Morse-Bott \\
            \hline
            $Z$ & $0$ & $S^{p_1-1}\times S^{p_2-1}$ & $(1-t^{p_1-1}) (1-t^{p_2-1 )})$\\
  $X_1$ & $p_1-1$ & $ S^{p_2-1}$ & $    (1-t^{p_2-1}) t^{p_1-1} $\\
    $X_2$ & $p_2 - 1$ & $S^{p_1-1}$ & $    (1-t^{p_1-1}) t^{p_2-1}$\\
      $X_3$ & $p_1+p_2-2$ & $\bullet$  & $  t^{p_1+p_2-2} $\\
        $X_4$ & $p_1+p_2-2$ & $\bullet$ & $t^{p_1+p_2-2}$\\
        \hline
  \end{tabular}
\end{minipage}
\vspace{-0.5cm}
\caption{$Z$ lies in the interior.}
\label{fig:second}
\end{figure}

 \noindent
It follows that the Morse-Bott polynomial can be written as:
$$ P_{\mathbb{D}}(t) =1 + (1+t)(t^{p_1-2}+t^{p_2-2}) + (1+t)^2 t^{p_1+p_2-4} + t^{p_1+ p_2 -2} .$$

As a consequence: The Euler characteristic is $1 + (-1)^ {p_1+ p_2 -2}$, which is 2 if $p_1+p_2$ is even and 0 else.
The Morse-Bott inequalities imply that several $b_i= 0$  and the others are bounded by 1 or 2. But we can do better by another choice of $Z$.
\medskip

The second choice is the point $Z_1$ is outside the work space and chosen in such a way that critical points of type (1) do not occur.
We have now only critical points of type (2), which are isolated:  one minimum and one maximum with Morse  polynomial $P_{\mathbb{D}}(t) = 1 +  t^{p_1+ p_2 -2}$. Indeed since any Morse function with two critical points can only occur on a sphere, we have in in our case $S^{p_1+p_2-2}$.

\begin{figure}[ht]
\begin{minipage}[t]{0.3\linewidth}
\centering
\includegraphics[
                 width=0.95\linewidth]{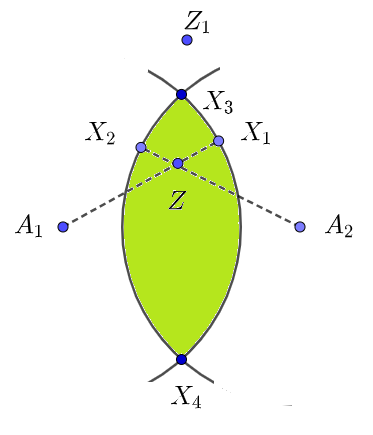}
\end{minipage}
\begin{minipage}[t]{0.65\linewidth}
\vspace{-4cm}
\begin{tabular}{|c|ccc|}
\hline
            & $index$ & $\Sigma$ & Morse-Bott \\
            \hline
  $X_3$ & $0$ & $\bullet$ & $1$\\
    $X_4$ & $p_1+p_2-2$ & $\bullet$  & $t^{p_1+p_2-2}$\\
  \hline
  \end{tabular}
\end{minipage}
\vspace{-1.5cm}
	\caption{$Z_1$ lies in the exterior.}
	\label{fig:exterior}
\end{figure}
NB. This answer can be checked in another way since our spider with 2 legs is also a polygon space.
We can also apply Proposition \ref{p:eulerchar} to compute the Euler characteristic. This gives: \\
$\chi(\mathcal{S}) = 2 - \chi(S^{p_1-1}) - \chi(S^{p_2-1}) + \chi(S^{p_1-1}) \chi(S^{p_2-1}) = 1 + (-1)^{p_1 +p_2 -2}.$

\medskip

We next consider a similar case, which includes a critical arc of non-extremal type. We get the table:

\begin{figure}[h!]
\begin{minipage}[t]{0.3\linewidth}
\centering
\includegraphics[
                 width=0.95\linewidth]{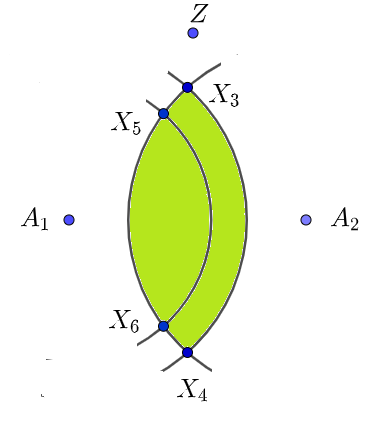}
\end{minipage}
\begin{minipage}[t]{0.65\linewidth}
\vspace{-4cm}
\begin{tabular}{|c|ccc|}
\hline
            & $index$ & $\Sigma$ & Morse-Bott \\
            \hline
  $X_3$ & $0$ & $\bullet$ & $1$\\
    $X_4$ & $p_1+p_2-2$ & $\bullet$  & $t^{p_1+p_2-2}$\\
      $X_5$ & $1$ & $\bullet$  & $t$\\
        $X_6$ & $p_1+p_2-3$ & $\bullet$ & $t^{p_1+p_2-3}$\\
        \hline
  \end{tabular}
\end{minipage}
\end{figure}

\FloatBarrier
  
 The Morse Bott polynomial  is $1 + t + t^{p_1+p_2-3} + t^{p_1 +p_2-2}$ \\
Via `Morse-Bott-inequalities' it follows that in this example the squared distance is a perfect Morse function ($p_1+p_2 \ge 6$).

\section{Hooke Energy}\label{s:hooke}

 This was our starting point  in this study of spiders. The case of 3 legs with two edges each has been studied before in \cite{KS}. It will turn out that weighted Hooke energy and squared distance give equivalent results.

\smallskip
\noindent
The Hooke Energy  potential $H : \mathbb{R}^2 \to \mathbb{R}$ is defined by
$  H(x) = \sum_{i=1}^n || x- a_i||^2 $ .\\
The stationary points of $H$ are determined by: 
$\nabla H = n x - 2 \sum_{i=1}^n || x- a_i||= 0$, so $x =\frac{1}{n} \sum_{i=1}^n a_i$, the centre of gravity.
There are no other stationary points in $\mathbb{R}^2$.
Note that $H(X) = n ||x-z||^2 +  \kappa,$ where $z=\frac{1}{n} \sum_{i=1}^n a_i$, the centre of gravity $Z$ and $\kappa= \sum_{i=1}^n ||a_i||^2 - n ||z||^2$. All level curves are circles.

\begin{proposition}\label{p:hooke}
The Hooke energy $H$ is an affine linear  transform of the squared distance function to the centre of gravity $Z$.
$$ H(x) = n|| x-z||^2 + \kappa .$$ The Hooke function is minimal in the centre of gravity.
Both functions have the same critical point theory.
\end{proposition}

Given weights $w = w_1,\ldots, w_n$ to the Hooke potential, we get a weighted Hooke potential
$$  H_w(x) =  \sum_{i=1}^n w_i || x- a_i||^2 .$$ 
The weighted centre of gravity $Z_w$ is defined by  $z_w=\frac{1}{\sum w_i} \sum_{i=1}^n w_i a_i$.
NB. We will allow also negative weights. We extend the definitions to spider space: 
$\mathbb{H}(\xi) = H(\pi(\xi)) $ and similarly $\mathbb{H}_w(\xi) = H_w(\pi(\xi)) $. 

\begin{theorem}\label{t:whooke}
Weighted Hooke Energy and Squared Distance to the weighted centre of gravity have equivalent critical point theory on spider space $\mathcal{S}$. 
\end{theorem}
\begin{proof}

We do a similar computation as above:
$$H_w(x) = (\sum_{i=1}^n w_i)  ||x-z_w||^2 +  \kappa_w$$
where   
$\kappa_w= \sum_{i=1}^n w_i ||a_i||^2 -(\sum w_i )||z_w||^2$.
Weighted Hooke Energy is therefore also an affine linear transform of the squared distance function.
\end{proof}
NB. Depending on the sign  of $w_i$ the indices of critical points can become dual.

\section{Voronoi distance}\label{s:voronoi}

\subsection{Critical points}

In this section we will consider on the generic smooth spider space the {\it Voronoi distance function} $\min_{i=1}^n||x-a_i||$. To smoothen it at the feet $A_i$ and get a nicer gradient without modifying the level sets we will in fact work with 
$$
V(x)=\min_{1 \le i \le n}||x-a_i||^2=\mathrm{dist}(x,\{A_1,\dots, A_n\})^2.
$$
This function is defined on the spider space $\mathcal{S}$, i.e. $\mathbb{V}(\xi):=V(\pi(\xi))$.
A major role is naturally played here by the {\it Voronoi diagram} of the feet $A_1,\dots, A_n\in{\Rz}^2$, i.e. the set 
$$
\mathcal{V}:=\{x\in{\Rz}^2\mid \exists i,j\colon i\neq j, \textrm{and}\ V(x)=||x-a_i||^2=||x-a_j||^2\}
$$
and its intersection with the work space: $\mathcal{V} \cap \mathcal{W}$. It is in fact the set of non-differentiability points of $V(x)$. The function itself being locally Lipschitz, has everywhere a well-defined Clarke subdifferential $\partial V(x)$ and thus the notion of critical point makes sense for $V(x)$, i.e. $x$ is critical iff $0\in\partial V(x)$ which is equivalent to saying that $x$ belongs to the convex hull of those feet $A_i$ that realize the distance $V(x)$. Otherwise, the point $x$ is called regular.

Note that $\mathbb{V}(\xi)$ is locally Lipschitz, too and thus we can carry over the notion of the Clarke subdifferential to this situation. This is intuitively clear, but needs a precise definition and extra explanation, see the Appendix. Note that $\mathbb{V}$ is differentiable outside $\pi^{-1}(\mathcal{V} \cap \mathcal{W})$.

\medskip

The critical point theory of $V(x)$ in the plane is well understood, see \cite{Dirk}. Obviously all the feet $A_i$ are global minima of the Voronoi potential $V(x)$. However, when moving to the spider space the situation presents some interesting features, the more so since there is some additional boundary information coming along.

\begin{figure}[ht]
\includegraphics[width=8cm]{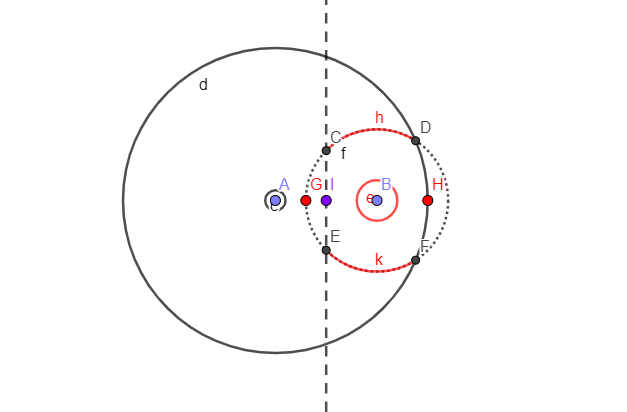}
\caption{The figure is self-explanatory: a two-legged spider with $\mathcal{W}$ being an annulus centered at $A$ and with several critical circles to be taken into account. The Voronoi diagram is the dashed line.}
\end{figure}

We first recall the Delaunay and Voronoi tessellation of the affine plane. There are several ways to define these. Here we choose the  \textit{void circle property} (vc):
A point $Y$ \textit{has the $k$-th vc-property} if there exist a circle with centre $Y$, which contains $k$ points from among  $A_1,\dots,A_n$ and no points from the feet in its interior. The centres of circles form the \textit{$k$-th void circle set}.  \textsl{We will assume as genericity condition that no points satisfy the 4-vc condition}, but remark that one can adapt the theory to a more general setting.

\smallskip
First we define the \textit{Delaunay tessellation} $DT(\mathcal{V})$: the 0-cells are the points $A_i$, the 1-cells are the 
edges which connect 
those vertices $A_i, A_j$ that are related to a point satisfying the 2-vc condition, the 2-cells are the cells characterized by triples related to points having the 3-vc property. Note that these  are closed sets.

\smallskip
Each foot $A_i$ defines a 2-dimensional Voronoi cell or {\it power cell} $$P(A_i):=\{x\in{\Rz}^2\mid ||x-a_i||^2\leq V(x)\},$$ its open part will be denoted by 
$P^{\circ}(A_i)$. The 1-dimensional Voronoi cells are the nonempty intersections $P(A_i) \cap  P(A_j)$ for $i\neq j$, the 0-cells are the triple intersections  $P(A_i) \cap  P(A_j) \cap P(A_k)$.
This tessellation is dual to the Delaunay tessellation.

\begin{theorem} [e.g. \cite{MS}]
Critical points (in the topological sense) of the Voronoi distance  function are precisely the intersections between a Delaunay-cell and its dual Voronoi cell. The index of the critical point is equal to the dimension of the Delaunay cell.
\end{theorem}

It follows that all the points $A_i$ are minima, the intersection between Voronoi-edges and Delaunay-edges are topological saddle points and the centers of 3-void-circles are maxima; we shall use the respective notations: $m,S,M$. These last two types of critical points, saddles and maxima, are contained in the Voronoi diagram $\mathcal{V}$.

\begin{lemma}\label{unique-cell}
For a configuration $\xi\in\mathcal{S}$ such that $X=\pi(\xi)\notin \mathcal{V}$, there exists exactly one power cell such that $X\in P^{\circ}(A_i)$ and thus $\mathbb{V}(\zeta)=||\pi(\zeta)-a_i||^2$ for $\zeta\in \pi^{-1}(P^{\circ}(A_i))$.
\end{lemma}

 Since $\mathbb{V}$ is the composition of the work space map and the Voronoi Distance, it is good to consider its derivative
  $ d_\xi\mathbb{V}$, which is well-defined at a regular spider $\xi$  (i.e. with $\pi(\xi)\notin \mathcal{V}$)
  as a composition of
    $$  d_\xi\pi: T_{\xi}\mathcal{S} \to T_{\pi(\xi)} \mathbb{R}^2  \; \mbox{and} \;  d_xV: T_{\pi(\xi)} \mathbb{R}^2 \to \mathbb{R}.$$

\begin{proposition}[Composition of regular maps]
If $\xi\in\mathcal{S}$ is a regular point of the work map $\pi$,  and $V$ is regular at $X=\pi(\xi)$, then $\mathbb{V}$ is regular at $\xi$.
\end{proposition}
The proof is in the Appendix. This Proposition tells us that anyhow we have regularity of $\mathbb{V}$, provided $X$ is not one of the critical points of $V$ nor lies on a critical circle. We are left with  two cases for a critical point $\xi$ of $\mathbb{V}$:
\begin{enumerate}
\item $\xi$ projects to a critical point of $V$,
\item $\xi$ projects to a point lying on a critical circle.
\end{enumerate}

\medskip

{\bf Case (1): } Since $\pi$ is a proper mapping with continuously varying fibres, it is rather easy (see the Appendix)  to see that $0\in \partial V(\pi(\xi))$ implies $0\in \partial \mathbb{V}(\xi)$. This together with the previous observations allows us to state the next result whose proof is straightforward.
\begin{proposition} \label{p:V-sing}
Assume $X=\pi(\xi)$ is a critical point of $V$, then so is $\xi$ for $\mathbb{V}$. If moreover $X$ does not lie on a critical circle,
then $\pi^{-1}(X) $  is a smooth component of the critical set of $\mathbb{V}$ of codimension 2, satisfying the Morse-Bott property with transversal type, equal to the singularity type of $V$ at $X$ (minimum, saddle or maximum).
\end{proposition}

{\bf Case (2): }
 Inside the $i$-th power cell the function $V$ reduces to the distance to the unique foot $A_i$, and one applies Theorem \ref{t:critical} with $Z=A_i$. But that theorem assumes that the foot is different from $Z$ by the $(\star)$-condition. Therefore we have to consider two cases.
 On the open part of Voronoi cells of dimension 2 we consider in the work space all the non-void intersections with critical circles: $C(A_j, r_j^{\epsilon_j}) \cap P^{\circ}(A_i) \cap \mathcal{W}$. The case $i \ne j$ is covered by Theorem \ref{t:critical}, the case $i=j$ gives rise to non-isolated singularities in the work space, which are discussed in the Section \ref{ss:robotarm} about robot arms.
\begin{proposition}\label{p:open-cell}

 For a generic smooth spider space and the Voronoi distance $\mathbb{V}$, a configuration $\xi\in\mathcal{S}$ with $\pi(\xi) \ne A_j$ is critical if and only if one of the following conditions holds for $X=\pi(\xi)$ lying in the open part of the of 2-dimensional cell  $P^{\circ}(A_i)$ :
    \begin{enumerate} 
    \item in case $i \ne j$: If one leg is aligned, $X$ is the intersection point $C(A_j,r_j^{\epsilon_j}) \cap A_iA_j \cap \mathcal{W} $;
    \item in case $i \ne j$:   If two legs are aligned -- the points  $X$  of intersections  of 2 critical circles;
\item    in case $i=j$ we get all the intersections $C(A_i, r_i^{\epsilon_i}) \cap P^{\circ}(A_i) \cap \mathcal{W}$. These are generically arcs or circles.
    \end{enumerate}

\end{proposition}

    \begin{remark}\label{r:open-cell}
      In part (1) of Proposition \ref{p:open-cell}  the set $ C(A_j, r_j^{\epsilon_j}) \cap A_iA_j\cap P^{\circ}(A_i) \cap \mathcal{W}$ consists of isolated points. The corresponding critical set in $\mathcal{S}$ has the Morse-Bott property. The index follows from Theorem \ref{t:indices} (2). Similarly, in (2) the Morse-Bott index follows from Theorem \ref{t:indices}(3). In case (3) 
      the intersections $C(A_i, r_i^{\epsilon_i}) \cap P^{\circ}(A_i) \cap \mathcal{W}$ consist of non-isolated pieces (arcs or circles). The transversal type at the intersection points is generically well defined and gives us locally the Morse-Bott  property with index as in Proposition \ref{p:robotarm} about robot arms. But in general the function does not have the global Morse-Bott property. 
 If $X=A_j$, then we are in Case (1): Morse-Bott with transversal minimum.
      
In the case of spiders where each leg has only 2 edges, the description becomes more transparent. See the Example in Figures \ref{fig:voronoi} and \ref{fig:voronoiCount}.
    \end{remark}

After studying the power cells  we study next what happens with $\mathbb{V}$ at the points of $\pi^{-1}(\mathcal{V} \cap \mathcal{W})$. We are left with the following two cases:
\begin{itemize}
\item the inverse images of the critical points of $V$, studied above in Proposition \ref{p:V-sing},
\item the inverse images of the intersections of critical circles with the Voronoi diagram.
\end{itemize}
and of course combinations of them.

\medskip
Assume a point $X$ belongs to $\mathcal{V}$, but is not critical for $V$. Let us introduce the following notation: if the two vectors $a-x$ and $b-x$ are linearly independent, then we say that the corresponding points $A,X, B$ define the angle $\angle AXB$ at $X$ defined to be the acute one. To this angle corresponds an affine cone denoted by $<AXB>$ obtained as the union of all those affine lines through $X$ that meet the segment $AB$. If now $X$ lies on a circle $\mathcal{K}$ centred at a point $C\neq A,B$, then we have the obvious observation that the existence of a point $D$ in the segment $AB$ for which $XD$ is perpendicular to the tangent line to $\mathcal{K}$ at $X$ is equivalent to $XC$ defining a line from the cone $<AXB>$, i.e. $XC$ lying inside this cone. We will use this in the next Proposition.
\begin{proposition}\label{p:vor-critcircle}
Assume  a point $X \in \mathcal{V}$ lies on a single critical circle $C(A_k,r)$ and on the Voronoi edge of $A_iA_j$, but is not a critical point of $V$. Let $<A_iXA_j>$  be the affine cone around the angle $\angle A_iXA_j$, then 
\begin{enumerate}
\item  Assume $ i\ne k \ne j$ then
\begin{itemize}
\item If $XA_k$ lies in the interior of $<A_iXA_j>$ then $\pi^{-1} (X)$ is a critical component of $\mathbb{V}$, with transversal Morse index, equal to the Morse index of $V$ at $X$,
\item  If $XA_k$ lies in the exterior of $<A_iXA_j>$ then $\pi^{-1} (X)$ does not contain critical points, in fact $\mathbb{V}$ is topologically regular near $\pi^{-1} (X)$,
\end{itemize}

\item Assume $ i\ne k = j$ then $\mathbb{V}$ is singular at $\pi^{-1} (X)$ and a limit of the case (3)  in Proposition  \ref{p:open-cell}.
\end{enumerate}
\end{proposition}
\begin{proof}
We follow the proof of Theorem \ref{t:critical}; by the formula $(*)$ in the Appendix, any $\ell\in \partial \mathbb{V}(\xi)$ is a composition of 
    $$  d_\xi\pi: T_{\xi}\mathcal{S} \to T_{\pi(\xi)} \mathbb{R}^2  \; \mbox{and} \;  L\in \partial V(\pi(\xi)).   $$
    
Here $L$ is a submersion (see Remark \ref{submersion} in the Appendix). 
     Let  $ i\ne k \ne j$. If $\xi$ does not belong to the critical locus of $\pi$ then we get the composition of two submersions, so $\mathbb{V}$ is regular at $\xi$ (see also Lemma \ref{concluding} in the Appendix). As soon as $\xi$ has exactly one aligned leg, then the image of $d_\xi\pi$ is one-dimensional. The normal to the tangent line to the critical circle at $X$ must be contained in the cone   $<A_iXA_j>$  (which is actually related to the Clarke subdifferential)  to get a critical point of $\mathbb{V}$.

The case $ i\ne k = j$ is left to the reader.
\end{proof}

Combinations of these cases can be avoided by genericity.
We consider a very strong type of genericity,  adapted to the Voronoi diagram: \textit{strongly Voronoi genericity} which means:
\begin{itemize}
\item  we have a {\it strongly generic} spider space,
\item the  critical points of $V$ do not lie on critical circles,
\item there are no 4-vc critical points,
\item no configuration with the body lying on the Voronoi set has two aligned legs.
\end{itemize}

Taking into account all the previous discussion, we may summarize the results  of Propositions \ref{p:V-sing},  \ref{p:open-cell} and \ref{p:vor-critcircle} in the following form:

\begin{theorem}\label{VorCrit}
 For a strongly Voronoi generic spider space and the Voronoi distance $\mathbb{V}$, a configuration $\xi\in\mathcal{S}$ is critical if and only if one of the following conditions hold for $X=\pi(\xi)$:
 \begin{enumerate}
\item In the case when there are no aligned legs: $X$ is a critical point of the Voronoi distance $V$,
\item In the case of one aligned leg: 
\begin{enumerate}
\item in an open power cell \\ -- $X$ is the intersection point $C(A_j,r_j^{\epsilon_j}) \cap A_iA_j \cap \mathcal{W} $ if $i \ne j$,and \\ -- $X$  lies on the critical circle if $i=j$ (generically arcs or the full circle),
\item on the Voronoi diagram -- if X satisfies the condition: $XA_k$ lies in the interior of $<A_iXA_j>$ ,

\end{enumerate}

\item in case of two aligned legs: $X$ is the intersection point of critical circles.

\end{enumerate}
\end{theorem}

Figure \ref{fig:voronoi} illustrates the general situation.
\begin{figure}[h]
\includegraphics[width=8cm]{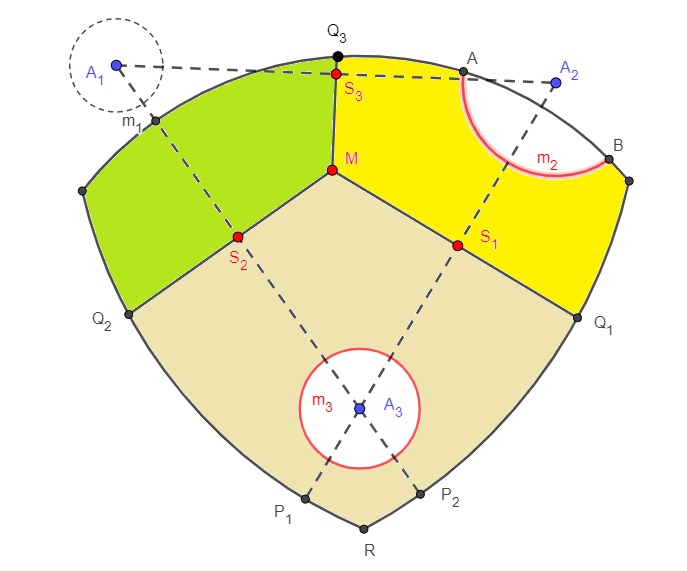}
\caption{The Voronoi diagram is an inverse $Y$-fork. Its trace $\mathcal{V} \cap \mathcal{W}$ on the work space gives three cells. We can see inside $\mathcal{W}$ three saddle points $s_i$, one maximum $M$, non-isolated minima on the critical circles $m_2$, $m_3$ while the first critical circle lying outside the work space contributes only by one point. Moreover, there are additional local maxima $Q_1, Q_2, Q_3$ and $R$, while $P_1, P_2$ are additional saddle points. Inside each cell the theory is covered by Theorem \ref{t:critical}.}
\label{fig:voronoi}
\end{figure}


\subsection{Morse-Bott type polynomials for the Voronoi distance}
Generating polynomials for singularities of functions have been studied by describing the changes in topology when passing a critical level. This is well understood in the case of Morse functions in the isolated case and Morse-Bott functions in the non-isolated case. For other types of (transversal) singularities and other critical sets there should exist anyhow generalisations. In the preceding section we saw that for the Voronoi distance complicated situations can occur, e.g. non-isolated singularities in work space (even not necessarily smooth) will get us outside the theory of Morse-Bott functions.

We propose to compute a Morse-Bott polynomial by using a small perturbation so that we avoid non-isolated singularities in the work space. This perturbation will act as a  `Morsification'.

Take points  $\bar{A_i}$ near to $A_i$ and consider for the same spider the function
$$
\bar{V}(x)=\min_{1 \le i \le n}||x-\bar{a_i}||^2=\mathrm{dist}(x,\{\bar{A_1},\dots, \bar{A_n}\})^2,
$$
and its lifting to the spider space.
The new Voronoi diagram will be denoted by $\bar{\mathcal{V}}$.

For generic positions of the points $\bar{A_i}$  the statements of Theorem \ref{VorCrit} with only isolated intersections are applicable.
More precisely: in the work space the non-isolated contributions are replaced by the intersections of $A_i\bar{A_i}$ with the critical circles.
As a result $\bar{\mathrm{V}}$ becomes a Morse-Bott function, where the index can be computed via Theorem \ref{t:indices}  for  $\mathbb{D}=||x-z||^2$.

\begin{figure}[ht]
\includegraphics[width=8cm]{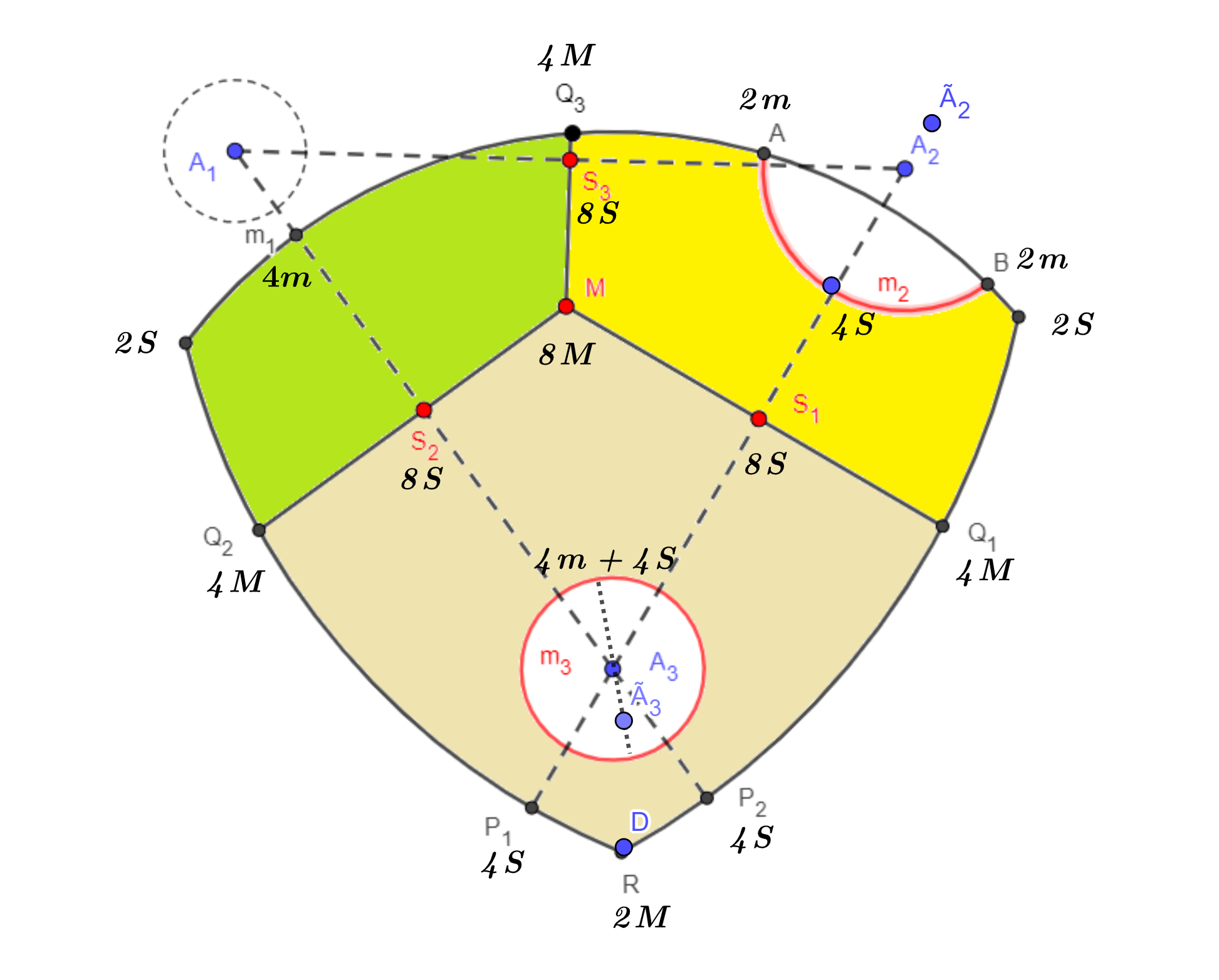}
\caption{Counting critical points in the example of Figure \ref{fig:voronoi}; isolated in blue; in red the contributions from the approximation. }
\label{fig:voronoiCount}
\end{figure}

\vspace{-0.3cm}
Let us show what happens in the tripod spider example of Figure \ref{fig:voronoi} with only 2 edges for each leg.
Note that  $\chi({\mathcal{S})} = -10$.  
The isolated critical points minima $[m]$, saddles $[S]$  and maxima $[M]$  are indicated in Figure \ref{fig:voronoiCount} combined with the covering degree. We write $4[m]  + 36 [S] + 22 [M]$  as their formal sum, which
contributes to the Morse polynomial by   $4 + 36 t + 22 t^2$. Take a point $\bar{A_3}$  near to $A_3$. The critical circle is replaced by 2 critical points $4[m] + 4[S]$, which contribute $4 +4t$. Next consider 
$\bar{A_2}$  near to $A_2$ (e.g. on the interval $A_2A_3$). The critical arc  is now replaced by a minima and two corner points:
$4[m] + 4[S]$,  which contributes $4 + 4t$ to the Morse polynomial of the perturbed Voronoi-distance. 
This resulting Morse polynomial is:  $12 + 44 t + 22 t^2$. Note that the Morse polynomial depends on the chosen Morsification. E.g. other positions of $\bar{A_2}$ can give contribution $2+2t$ instead of $4 + 4t$.

\section{Concluding remarks}\label{s:cremarks}

Several generalizations of the critical point theory of the squared distance are possible. 

\subsection{Higher dimensional ambient space}
A next step could be the extension to spiders in $d$-dimensional space. The spider space, especially its smoothness has been considered by several authors, as mentioned in the introduction. Very few is known  about the topology of spider spaces, except in the case of an \textit{ n-arms machine}: each leg has 2 edges of equal length and the feet lie at the vertices of a regular polygon in the plane. Kamiyama and Tsukuda \cite{KT}  determined the homotopy type of the spider space of this kind.
An essential tool was the theory of co-limits. 

\smallskip
\noindent
In the general $d$-dimensional case the Morse theoretic approach via the squared distance function could be an interesting project.

\subsection{Graph-Linkages}

A next extension could  be other types of graph-linkages. A first example is a tripod spider platform, consisting of a  triangular platform, where each of the vertices is connected by a robot arm with a fixed feet (and similar for a massive n-gon platform).
The paper \cite{SSB2} contains results about the smoothness of the configuration space and Morse results for tripod spider platforms in the plane and legs with only two edges.

\subsection{Regions with stratified differential boundary}

A simpler problem, which is directly related to our spider study is to consider a region in the plane with stratified differentiable boundary together with the quadratic distance to a point in the plane.  In the case of planar spiders the (stratified) critical points of quadratic distance coincide with the images of the singular sets (via work map) of quadratic distance on spider space.
In \cite{KS} this was used to treat Hooke energy for tripods in the plane.

\section{Appendix}

In this appendix we treat  the definition and certain properties of the Clarke subdifferential on a smooth manifold. They look intuitively clear,
but need some careful attention.

\medskip

We noted that $\mathbb{V}(\xi)$ is locally Lipschitz and thus we can carry over the notion of the Clarke subdifferential to this situation. Clearly, $\mathbb{V}$ is differentiable outside $\pi^{-1}(\mathcal{V} \cap \mathcal{W})$.

From the general point of view, given a $\mathcal{C}^1$-smooth submanifold $M\subset{\Rz}^N$ of dimension $m$ and a locally Lipschitz function $h\colon M\to \mathbb{R}$, we obviously have that $h$ is almost everywhere differentiable along $M$ (w.r.t. the usual $m$-dimensional Hausdorff measure; it is enough to compose $h$ with local parametrizations and invoke the Rademacher Theorem). Therefore, for a point $a\in M$, we may repeat the classic Clarke approach by taking the convex hull, denoted again by $\partial h(a)$, of all the possible limits of gradients $\nabla h(x_\nu)$ when $x_\nu$ are points at which $h$ is differentiable and that tend to $a$. Here, $\nabla h(x)\in T_x M$ is the tangential component of the gradient $\nabla H(x)$ of any differentiable extension of $h$ onto a neighbourhood of $a$ (it does not depend on the choice of $H$). Note that the tangent spaces vary continuously, so that $\partial h(a)\subset T_a M$. This tangential gradient coincides with the unique $v\in T_xM$ such that $h(x')=h(x)+\langle v, \rho(x-x')\rangle+o(||x-x'||))$ where $\rho\colon \mathbb{R}^N\to T_xM$ is the orthogonal projection, while the inner product is the usual one in $\mathbb{R}^N$.

The point $a\in M$ is called critical, if $0\in\partial h(a)$. The notion of subdifferential can be also applied to vector-valued functions once we interpret their differentials as matrices. It also can be proved that given any local parametrization $\gamma$ of $M$ in a neighourhood of $a$, $\partial h(a)=d_{\gamma^{-1}(a) } \gamma(\partial(h\circ \gamma)(\gamma^{-1}(a)))$

Consider now a $\mathcal{C}^1$ map $g\colon M\to \mathbb{R}^n$ and a locally Lipschitz function $f\colon U\to\mathbb{R}$ where $U$ is a neighbourhood of $g(M)$. 

\begin{lemma}\label{two}
If $g$ is a submersion at $p\in M$, then $0\in \partial(f\circ g)(p)$ implies $0\in \partial f(g(p))$.
\end{lemma}
\begin{proof}
Assume for simplicity $g$ is everywhere a submersion. 
    By the Rademacher Theorem, $f$ is differentiable outside a zero measure set $Z\subset U$. The co-area formula (generalized Fubini Theorem) implies $S:=g^{-1}(Z)$ has measure zero, too, since $g$ is a submersion. Now, $0=\sum\lambda_i \ell_i$ is a convex combination of some $\ell_i\in \partial (f\circ g)(p)$ and each $\ell_i$ is the limit of some sequence $\ell_i^\nu:= d_{x_\nu^i}(f\circ g)$ where the points $x_\nu^i\in M\setminus S$ converge to $p$. Putting $y_\nu^i:=g(x_\nu^i)$, we get $d_{x_\nu^i}(f\circ g)=d_{y_\nu^i}f\circ d_{x_\nu^i}g$ where $d_{x_\nu^i}g\colon T_{x_\nu^i} M\to \mathbb{R}^n$ is surjective. The tangent planes converge to $T_p M$ while the corresponding differentials to $d_p g$, and $y_\nu^i$ to $g(p)$, for each $i$. Passing to a subsequence, if necessary, we may assume that $d_{y_\nu^i}f$ tend to some $\delta_i\in \partial f(g(p))$ (these sequences can be considered as elements of $\mathbb{R}^N$; the sequences are bounded since $f$ is Lipschitz). Then $0=\sum(\lambda_i\delta_i) \cdot d_pg$ after identification with matrices. In particular, $d_p g$ is an $n\times m$ matrix containing an $m\times m$ invertible submatrix $A$, whereas $\sum(\lambda_i\delta_i)$ is in fact a $1\times n$ matrix of `unknowns'. Then $0=\sum(\lambda_i\delta_i)\cdot A$ has zero as the unique solution.
\end{proof}  

\begin{lemma}\label{three}
    If $M$ is compact, $g$ is open onto its image $W=g(M)$, and $g(p)$ lies in the closure of the interior of $W$, then $0\in \partial f(g(p))$ implies $0\in \partial(f\circ g)(p)$.
\end{lemma}
\begin{proof}
    Similarly to the previous proof, we have $0=\sum\lambda_i \ell_i$ as a convex combination of matrices $\ell_i=\lim \nabla f(y_\nu^i)$ for some $y_\nu^i\in U\setminus Z$ converging to $g(p)$. Choosing arbitrarily $x_\nu^i\in g^{-1}(y_\nu^i)$ and passing to convergent subsequences, we may assume that $x_\nu^i$ converge to some $x^i$ and thus $g(x^i)=g(p)$. This allows us to write $\nabla (f\circ g)(x_\nu^i)=\nabla f(y_\nu^i)\cdot d_{x_\nu^i}g$ and these sequences converge to $\ell_i.d_{x^i}g$. In order to avoid the problem of obtaining different points $x^i\in g^{-1}(g(p))$ we observe that $g$ being open onto its image, it has continuously varying fibres (see e.g. \cite{DL} Lemma 3.8): $g^{-1}(y_\nu^i)$ converges to $g^{-1}(p)$ in the Hausdorff metric. This implies in particular that the point $p\in g^{-1}(g(p))$ can be obtained as the limit of some points $z_\nu^i\in g^{-1}(y_\nu^i)$. Then   $\nabla (f\circ g)(x_\nu^i)$ tends to $\ell_i\cdot d_p g\in \partial (f\circ g)(p)$ and $0=\sum \lambda_i(\ell_i\cdot d_p g)$.
\end{proof}

We observe that the work map, the spider space, and the Voronoi distance satisfy the assumptions of the Lemma above. Moreover, in our situation the spider space is an algebraic manifold and the work map is analytic. Hence, both the set of critical points of $\pi$ (outside of which $\pi$ is a submersion) as well as $\pi^{-1}(\mathcal{V})$ have measure zero and so -- just as in the classical case of functions defined in the Euclidean space -- in order to determine $\partial \mathbb{V}(\xi)$ it is enough to approximate $\xi$ by sequences $\xi_\nu$ at which $d_{\xi_\nu}\pi=\pi|_{T_{\xi_\nu}\mathcal{S}}\colon T_{\xi_\nu}\mathcal{S}\to T_{\pi(\xi_\nu)}\mathbb{R}^2=\mathbb{R}^2$ has rank 2 and $X_\nu:=\pi(\xi_\nu)\notin\mathcal{V}$. In particular, $d_{\xi_\nu}\mathbb{V}=d_{x_\nu}V\circ \pi|_{T_{\xi_\nu}\mathcal{S}}$ and $d_{x_\nu}V(x)=2\langle x_\nu-a_{i_\nu}, x\rangle$ where $A_{i_\nu}$ is the unique point realizing the Voronoi distance $V(\pi(\xi_\nu))$. Passing to a subsequence, we get the same $i_\nu=i$ for all $\nu$. 

Let us stress once again that any non-differentiability point $\xi_0$ of $\mathbb{V}$ yields $X_0:=\pi(\xi_0)\in \mathcal{V}$. Eventually, any $\ell\in\partial \mathbb{V}(\xi_0)$, is a linear map $\ell\colon T_{\xi_0}\mathcal{S}\to\mathbb{R}$ of the form 
$$\ell(s)=2\langle x_0-a, \pi(s)\rangle,\> s\in T_{\xi_0}\mathcal{S},\leqno{(*)}$$ 
where $A$ is taken from the convex hull of those feet $A_1,\dots, A_n$ that realise the distance $V(x_0)$. On the other hand, any map $\ell$ of the previous form belongs to $\partial \mathbb{V}(\xi_0)$, provided $X_0$ lies in the interior of $\mathcal{W}$. Indeed, we may define $L(x)=2\langle x_0-a, x\rangle$, $x\in \mathbb{R}^2$ and we know that $L\in \partial V(x_0)$. Thus, it is the convex combination of the limits $L_i$ of some $\nabla V(x_\nu^i)$ with $x_\nu^i\to x_0$ and all $X_\nu^i\in \mathcal{W}\setminus\mathcal{V}$, because $X_0$ is an interior point of the work space. The spider space being compact and the work map open, we easily produce the corresponding required sequences $\xi_\nu^i\to \xi_0$ to finally get the assertion. 

All this implies in particular the following two remarks gathered in one obvious Lemma:
\begin{lemma}\label{concluding}
 In the situation considered above and with its notations, if $\xi_0$ is a critical point of $\mathbb{V}$, then $x_0-a$ is orthogonal to $\pi(T_{\xi_0} \mathcal{S})$, for some $A$ in the convex hull of the feet realising the distance $V(x_0)$. On the other hand, any critical point $x_0\in\mathcal{V}$ of $V$ and lying in the interior of the work space, is the image of a critical point of $\mathbb{V}$ by the work map.
\end{lemma}
\begin{remark}\label{submersion}
According to the discussion just made, if $0\notin \partial V(x_0)$, then any $L\in \partial V(x_0)$ is a submersion. 
\end{remark}

We end with the two remaining statements in section \ref{s:voronoi}:
\subsection{Proposition 9}[Composition of regular maps]
If $\xi\in\mathcal{S}$ is a regular point of the work map $\pi$ and $V$ is regular at $X=\pi(\xi)$, then $\mathbb{V}$ is regular at $\xi$.

This is a straightforward consequence of Lemma \ref{two}.

\subsection{Case 1}  It is easy to see that $0\in \partial V(\pi(\xi))$ implies $0\in \partial \mathbb{V}(\xi)$.

Apply Lemma \ref{three}.

\end{document}